\documentclass[12pt,leqno]{compositio}

\usepackage{amsmath,amssymb,amscd,latexsym,color, amsfonts}

\usepackage[all]{xy}

\newcommand{\spsms}{\Delta^{op}\mathrm{Pre}(Sm/S)}
\newcommand{\spc}{\Delta^{op}\mathrm{Pre}({\mathcal C})}

\renewcommand{\H}{{\mathrm{H}}}
\newcommand{\SH}{{\mathrm{SH}}}

\newcommand{\colim}{{\mathrm{colim\,}}}
\newcommand{\hocolim}{{\mathrm{hocolim\,}}}
\newcommand{\Mor}{{\mathrm{Mor}}}
\newcommand{\Ho}{{\mathrm{Ho}}}
\newcommand{\Cyl}{{\mathrm{Cyl}}}

\newcommand{\U}{{\mathrm{U}}}

\newcommand{\N}{\mathbb{N}}

\newcommand{\G}{\mathbb{G}}

\newcommand{\A}{\mathbb{A}}

\newcommand{\HH}{{\mathrm{H}}}

\renewcommand{\G}{\mathbb{G}}

\newcommand{\Pa}{{\mathbb{P}}}

\newcommand{\Spc}{{\mathrm{Spc}}}

\newcommand{\Z}{{\mathbb{Z}}}

\newcommand{\Hom}{\underline{\mathrm{Hom}}}

\newcommand{\RHom}{{\mathrm R}\underline{\mathrm{Hom}}}

\newcommand{\map}{\mathrm{map}}

\newcommand{\Rmap}{\mathrm{Rmap}}

\newcommand{\Spec}{\mathrm{Spec}}

\newcommand{\Ch}{{\mathcal C}}

\newcommand{\Dh}{{\mathcal D}}

\newcommand{\Gh}{{\mathcal G}}

\newcommand{\Mh}{{\mathcal M}}

\newcommand{\Sh}{{\mathcal S}}

\newtheorem{theorem}{Theorem}

\newtheorem{lemma}[theorem]{Lemma}

\newtheorem{prop}[theorem]{Proposition}

\newtheorem{defn}[theorem]{Definition}

\begin{document}

\title[Brown representability in $\A^1$-homotopy theory]{Brown representability in $\A^1$-homotopy theory}

\author{Niko Naumann}
\email{niko.naumann@mathematik.uni-regensburg.de, markus.spitzweck@mathematik.uni-regensburg.de}
\address{NWF I- Mathematik\\ Universit\"at Regensburg\\93040 Regensburg}
\author{Markus Spitzweck}
\email{markus.spitzweck@mathematik.uni-regensburg.de}
\address{NWF I- Mathematik\\ Universit\"at Regensburg\\93040 Regensburg}



\begin{abstract} We prove the following result of V. Voevodsky. If $S$ 
is a finite dimensional noetherian scheme such that $S=\cup_\alpha\Spec(R_\alpha)$ for {\em countable} rings $R_\alpha$, then the
stable motivic homotopy category over $S$ satisfies Brown
representability.
\end{abstract}

\maketitle

\section{Introduction}\label{intro}

Let $S$ be a noetherian scheme of finite Krull dimension. F. Morel and V. Voevodsky
construct the stable motivic homotopy category $\SH(S)$ guided by the 
intuition that there is a homotopy theory of schemes in which the affine
line plays the role of the unit interval in classical homotopy theory.\\
It is important to know which of the familiar structural properties of
the classical homotopy category hold for $\SH(S)$. It is well known
that $\SH(S)$ is a compactly generated triangulated category but this is a
rather weak statement: Whereas the classical homotopy category admits 
the sphere spectrum as a single compact generator, $\SH(S)$ requires an infinite set 
of such generators.\\
A more subtle question is whether Brown representability
holds for $\SH(S)$, cf. \cite{Neeman} for a discussion of this notion and
the fact that the following main result implies Brown representability for
$\SH(S)$. Denote by $\SH(S)^c\subseteq \SH(S)$ the full subcategory
of compact objects.

\begin{theorem}\label{1}
If $Sm/S$, the category of smooth $S$-schemes of finite type, is countable,
then so is $\SH(S)^c$.
\end{theorem}

This result is due to V. Voevodsky \cite[Proposition 5.5]{voevodsky}.\footnote{
In {\em loc. cit.} V. Voevodsky works with the hypothesis that there
is a Zariski open cover $S=\cup_\alpha\Spec(R_\alpha)$ for countable
rings $R_\alpha$. Given that $S$ is noetherian, in particular quasi-compact,
it is easy to see that this hypothesis is equivalent to $Sm/S$ being countable.}
We obtain it here from some unstable results which are of independent interest and which we now sketch, giving an outline of this paper:\\
In Section \ref{afg} we show that, for general $S$, the homotopy category of 
pointed motivic spaces $\H_\bullet(S)$ over $S$
admits an almost finitely generated 
and monoidal model.
The proof of this gives us a controlled fibrant replacement 
functor which allows to 
show in Section \ref{unstable} the following unstable finiteness result.

\begin{theorem}\label{2}
If $Sm/S$ is countable, $F\in\spsms_\bullet$ is sectionwise countable
and $(X,x)\in\spsms_\bullet$ is of finite type, then 

\[ \H(S)_\bullet((X,x),F) \]

is countable.
\end{theorem}

\newpage

Section \ref{stable} starts by establishing
\cite[Theorem 5.2]{voevodsky}:

\begin{theorem}\label{3} If $(X,x)\in\spsms_\bullet$ is of finite type and $E=(E_n)_{n\ge 0}$
is a $\Pa^1$-spectrum, then

\[ \SH(S)(\Sigma^\infty_{\Pa^1}(X,x),E)=\colim_n\,  \H_\bullet(S)((\Pa^1,\infty)^{\wedge n}\wedge (X,x),E_n).\]
\end{theorem}

Then we give the proof of Theorem \ref{1}.\\
We conclude the Introduction with a general remark concerning Brown representability
for $SH(S)$. As anticipated in \cite{voevodsky}, it is much more difficult to apply than its
classical counterpart. This is essentially because, given a ``cohomology theory''
on $Sm/S$ (of which there are many interesting examples), it is generally
difficult to extend it to $\SH(S)^c$, a minimum requirement for Brown representability to apply.\\
However, in recent joint work with P. A. {\O}stv{\ae}r \cite{nmp}, we constructed many
new motivic (ring) spectra, using the full strength of Theorem \ref{1}, and this was our initial
motivation for documenting its proof.

\begin{acknowledgements} 
Theorems \ref{1} and \ref{3} are due to V. Voevodsky but his proofs remain unpublished. In 
finding the present proofs via the above unstable results, we have been guided by
\cite[Section 4]{hovey},\cite{blander} and \cite{ppr}.

\end{acknowledgements}

\section{An almost finitely generated model for motivic spaces}\label{afg}

Let $S$ be a noetherian scheme of finite Krull dimension. The homotopy category of pointed
motivic spaces over $S$, denoted $\H_\bullet(S)$ \cite[Section 3.2]{mv}, is the homotopy
category of the category $\spsms_\bullet$ of pointed 
simplicial presheaves on the category $Sm/S$
of smooth $S$-schemes of finite type with respect to a suitable model structure.
The purpose of this Section is to show that this model structure can be chosen to
be monoidal and almost finitely generated, cf. \cite[4.1]{hovey}. \\
To set the stage, denote by $\Mh$ the model category $\spsms_\bullet$ 
with the objectwise
flasque model structure \cite[Theorem 3.7,a)]{isaksen}. We will show

\begin{prop}\label{afgen}
There exists a set $\Sh\subseteq\Mor(\Mh)$ such that the left Bousfield localization
$L_\Sh\Mh$ of $\Mh$ with respect to $\Sh$ exists, is a proper, cellular, simplicial, monoidal 
and almost finitely generated model category and satisfies $\Ho(L_\Sh\Mh)\simeq\H_\bullet(S)$.
\end{prop}

\begin{proof}
By \cite[Section 5]{isaksen} and \cite[Theorem A.3.11]{ppr}, for a suitable choice 
of $\Sh$ to be recalled presently, $L_\Sh\Mh$ exists, is proper, cellular and simplicial
and satisfies $\Ho(L_\Sh\Mh)\simeq\H_\bullet(S)$.\\
By \cite[Theorem 4.9]{isaksen}, we can choose

\[ \Sh=\{ \Cyl((U \amalg_{U\times_X V} V\to X)_+)\}_\alpha \cup \{ U_+\stackrel{0}{\to}
(U\times_S\A^1_S)_+\, |\, U\in Sm/S\}\]

where $0$ indicates the zero section and $\alpha$ runs through all elementary
distinguished squares \cite[Definition 3.3]{mv}

\[ \alpha\, =\, \xymatrix{Z\ar[r]^j\ar[d]& V\ar[d]\\U\ar[r]& X.}\]

We use \cite[Proposition 4.2]{hovey} to see that $L_\Sh \Mh$ is
almost finitely generated: $\Mh$ itself is proper and cellular by 
\cite[Theorem 3.7,a)]{isaksen} and finitely generated by inspection 
of the generating (trivial) cofibrations of $\Mh$, given in
\cite[Definition 3.2]{isaksen}. 
Actually, while {\em loc. cit.} claims ``cellular'', it only proves ``cofibrantly
generated'', so let us quickly explain why the additional properties
\cite[Definition 12.1.1,(1)-(3)]{hirsch} are true, i.e. $\Mh$ is cellular:
$(1)$ and $(2)$ are implied by $\Mh$ being finitely generated and $(3)$ says that
all cofibrations in $\Mh$ are effective monomorphisms. By \cite[Lemma 3.8]{isaksen}, every cofibration of $\Mh$ is an injective cofibration, i.e. a monomorphism
, and it is easy to see that all monomorphisms in $\Mh$ are effective.\\
Since $\Sh$ consists of cofibrations with compact\footnote{To avoid confusion arising from conflicting
terminology in the literature, we make precise that we call an object $c$
of a category $\Ch$ compact, if for all filtering colimits $\colim_i c_i$ which 
exist in $\Ch$, the canonical map of sets $\Ch(c,\colim_i c_i)\to \colim_i\, \Ch(c,c_i)$ is bijective.}
domains and codomains, $L_\Sh\Mh$ is almost finitely generated by \cite[Proposition 4.2]{hovey}.\\
To see that $L_\Sh\Mh$ is monoidal requires some argument, cf.
\cite[Section 6]{isaksen}: First, the $\A^1$-local {\em injective}
model structure on $\spsms_\bullet$ is monoidal since smashing with every
pointed simplicial presheaf preserves $\A^1$-weak equivalences 
\cite[page 27]{morelictp}. Also, $\Mh$ itself is monoidal by \cite[Proposition
3.14]{isaksen} (for the unpointed variant of $\Mh$)
and \cite[Proposition 4.2.9]{hoveybook} (for the passage from the
unpointed case to $\Mh$). Now, let $i$ and $j$
be cofibrations in $L_\Sh\Mh$. Since $L_\Sh\Mh$ and $\Mh$ have the same
cofibrations, the push-out product $i\wedge j$ is a cofibration in $L_\Sh\Mh$.
Assume in addition that one of $i$ and $j$ is acyclic, i.e. an $\A^1$-
weak equivalence. Then so is $i\wedge j$ by the above reminder on the injective
structure.
\end{proof}

\section{Unstable results}\label{unstable}

We employ the following notions of finiteness: A set $X$ is countable if
there is an injective map $X\to\N$. A simplicial set $X$ is countable if
the disjoint union $\cup_{n\ge 0} X_n$ is countable. A presheaf of
(pointed) simplicial sets on a category $\Ch$ is sectionwise countable if for all
$U\in\Ch$, the (pointed) simplicial set $X(U)$ is countable. A category $\Ch$
is countable if it is equivalent to a category $\Ch'$ such that the disjoint
union $\cup_{c_1,c_2\in\Ch'}\, \Ch'(c_1,c_2)$ is countable.\\

We will need the following application of the small object argument.

\begin{prop}\label{smallobj} Let $\Ch$ be a category, $\spc_\bullet$ the category
of pointed simplicial presheaves on $\Ch$ and $I\subseteq\Mor(\spc_\bullet)$ a subset such that
\begin{itemize}
\item[i)] $I$ is countable.
\item[ii)] For every $\alpha\in I$, the domain $d(\alpha)\in\spc_\bullet$ of $\alpha$
is compact.
\item[iii)] For every $\alpha\in I$ and $G\in\spc_\bullet$ sectionwise countable, the set $\spc_\bullet(d(\alpha),G)$ is countable.
\item[iv)] For every $\alpha\in I$ and $U\in\Ch$, the set $\spc_\bullet(U_+,c(\alpha))$
is countable, where $c(\alpha)$ denotes the codomain of $\alpha$.
\end{itemize}

Then every map $F\to G$ in $\spc_\bullet$ can be functorially factored into
$F\stackrel{\iota}{\to}F'\stackrel{\pi}{\to} G$ such that
\begin{itemize}
\item[a)] $\iota$ is a relative $I$-cell complex \cite[Definition 10.5.8, (1)]{hirsch}.
\item[b)] $\pi$ has the right lifting property with respect to $I$.
\item[c)] If $G=\bullet$ is the final object and $F$ is sectionwise countable,
then $F'$ is sectionwise countable.
\end{itemize}
\end{prop}

\begin{proof} The small object argument \cite[Proposition 10.5.16]{hirsch} applies by $ii)$,
yielding a functorial factorization satisfying $a)$ and $b)$.\\
To see $c)$, we assume $F$ sectionwise countable and run through this 
argument in some detail:\\
We construct

\[ F=F_0\to F_1\to\ldots\to F_n\to F_{n+1}\to\ldots\]

by induction on $n\ge 0$ such that all $F_n$ are sectionwise countable
as follows: Consider the set $\Dh$ of all commutative squares

\[ \xymatrix{ a\ar[r]\ar[d]^f & F_n \ar[d]\\ b\ar[r] & \bullet } \]

with $f\in I$. Then $\Dh$ is countable by $i),iii)$. Define $F_{n+1}$ to
be the push-out

\[ \xymatrix{ \amalg_\Dh\, a \ar[r]\ar[d] & F_n\ar[d] \\ \amalg_\Dh\, 
b \ar[r] & F_{n+1}.} \]

Then $F_{n+1}$ is sectionwise countable by $iv)$.\\
Now let $F\stackrel{\iota}{\to} F':=\colim_n F_n\stackrel{\pi}{\to}\bullet$
be the canonical maps. These satisfy $a)$ trivially and $b)$ by $ii)$ (i.e. we do not need
any longer transfinite compositions). Clearly, $F'$ is sectionwise countable since all the $F_n$ are.
\end{proof}

When coupled with the work from Section \ref{afg}, this yields our following
key technical finiteness result.

\begin{prop}\label{fibrep} Let $S$ be a noetherian scheme of finite Krull dimension such that $Sm/S$ is countable and $F\in\spsms_\bullet$ sectionwise countable.
Then there is a trivial cofibration

\[ \xymatrix{F\,  \ar@{^(->}[r]^{\sim} & F'} \]

in $L_\Sh\Mh$ such that $F'$ is fibrant and sectionwise countable.
\end{prop}

\begin{proof} We apply Proposition \ref{smallobj} with $\Ch:=Sm/S$ and
$I:=\Lambda(\Sh)\cup J$ where $J$ is the set of generating trivial
cofibrations for $\Mh$ given in \cite[Definition 3.2,1)]{isaksen}, $\Sh$ is as in the proof of Proposition \ref{afgen} and

\[ \Lambda(\Sh):=\{ a\otimes\Delta^n \amalg_{a\otimes\partial\Delta^n}
b\otimes\partial\Delta^n\to b\otimes\Delta^n\, |\, (a\to b)\in\Sh, n\ge 0\}.\]

We check the assumptions $i)-iv)$ of Proposition \ref{smallobj}
for this set $I$:\\
$i)$ $I$ is countable since $Sm/S$ is.\\
$ii)$ We already know that all domains of $\Sh$ are compact, hence so are those
of $\Lambda(\Sh)$. Obviously the domains of $J$ are compact, so all domains
of $I$ are compact.\\
$iii)$ Assume $G\in\spsms_\bullet$ sectionwise countable and $\alpha\in I$.
To see that $T:=\spsms_\bullet(d(\alpha),G)$ is countable, we distinguish two cases:
If $\alpha\in J$, then 

\[ d(\alpha)=(\cup{\mathcal U}_+\wedge \Delta^n_+)\amalg _{\cup{\mathcal U}_+\wedge \Lambda^n_{k,+}}(X_+\wedge \Lambda^n_{k,+})\]

for a finite collection ${\mathcal U}=\{ U_i\to X\}_{U_i,X\in Sm/S}$ of monomorphisms.
Now, $T$ is countable since $G$ is sectionwise countable. If $\alpha\in\Lambda(\Sh)$, then 
$d(\alpha)=a\otimes\Delta^n \amalg_{a\otimes\partial\Delta^n}
b\otimes\partial\Delta^n$ for some $(a\to b)\in \Sh$, and it suffices 
to see that $\spsms_\bullet(a,G)$ and $\spsms_\bullet(b,G)$ are countable. Since $a$ and $b$
arise from representable presheaves by taking finite colimits and tensors with finite simplicial sets, this follows again from $G$ being sectionwise countable.\\
$iv)$ Assume $\alpha\in I$ with codomain $c(\alpha)$ and $U\in Sm/S$. We need
to see that $T:=\spsms_\bullet(U_+,c(\alpha))$ is countable: For $\alpha\in J$
we have $c(\alpha)=X_+\wedge\Delta^n_+$ for some $X\in Sm/s,n\ge 0$, 
hence the result since $Sm/S$ is countable.
For $\alpha\in\Lambda(\Sh)$ we have $c(\alpha)=b\otimes\Delta^n$ for some
$b=c(\beta),\beta\in\Sh$, hence $T=(b(U)\wedge\Delta^n_+)_0$. By construction of 
$\Sh$, $b$ is a finite push-out of tensors of finite simplicial sets with 
representables, so $T$ is countable since $Sm/S$ is.\\
Applying now Proposition \ref{smallobj},c), we obtain maps $F\stackrel{\iota}{\to}F'\stackrel{\pi}{\to}\bullet$ in $\spsms_\bullet$
such that $F'$ is sectionwise countable, $\iota$ is a relative $I$-cell
complex and $\pi$ has the right-lifting property with respect to $I$.\\
We need to check that $\iota$ (resp. $\pi$) is a trivial cofibration
(resp. a fibration) in $L_\Sh\Mh$: $J$ is a generating set of trivial cofibrations
for $\Mh$ and since $\Sh$ consists of cofibrations with cofibrant domain
in $\Mh$, $\Lambda(\Sh)$ consists of trivial cofibrations in $L_\Sh\Mh$.
Hence every relative $I$-cell complex, in particular $\iota$,
is a trivial cofibration in $L_\Sh\Mh$.
Since $\pi$ has the right-lifting property with respect to $J$,
$F'$ is fibrant in $\Mh$. Since $\pi$ has the right-lifting property with
respect to $I=\Lambda(\Sh)\cup J$, $F'\in\Mh$ is $S$-local \cite[Proposition 4.2.4]{hirsch}.
So $F'$ is fibrant in $L_\Sh\Mh$ by \cite[Proposition 3.3.16, (1)]{hirsch}.
\end{proof}

The previous result admits the following immediate stable analogue.

\begin{prop}\label{stablefibrep}
Let $S$ be a noetherian scheme of finite Krull dimension 
such that $Sm/S$ is countable and $E=(E_n)$ a $\Pa^1$-spectrum
\cite{jardine} such that all $E_n$ are sectionwise countable. Then 
there is a level-fibrant replacement $E'=(E'_n)$ of $E$ such that 
all $E'_n$ are sectionwise countable. 
\end{prop}

\begin{proof} One constructs $E'_n$ and structure maps
$E'_n\wedge \Pa^1 \to E'_{n+1}$ inductively: $E'_0:=E_0^f$, where
$(-)^f$ denotes the fibrant replacement provided by Proposition
\ref{fibrep}, and for all $n\ge 0$

\[ E'_{n+1}:=(E'_n\wedge\Pa^1\cup_{E_n\wedge\Pa^1} E_{n+1}\wedge\Pa^1)^f.\]

The evident maps $E'_n\wedge\Pa^1\to E'_{n+1}$ define a $\Pa^1$-spectrum $E'$
and the obvious map $E\to E'$ is a level-equivalence, $E'$ is level
fibrant and all $E'_n$ are sectionwise countable by an inductive application of
Proposition \ref{fibrep}.
\end{proof}

We recall (the simplicial variant of) the following definition
from \cite{voevodsky}.

\begin{defn}\label{ft} 
Let $S$ be a noetherian scheme of finite Krull dimension.
\begin{itemize}
\item[i)] The category of {\em motivic spaces of finite type over $S$}
is the smallest strictly full subcategory $\Spc^{ft}\subseteq\spsms$
such that:
\begin{itemize}
\item[a)] $(Sm/S)\subseteq\Spc^{ft}$.
\item[b)] For all push-outs

\[ \xymatrix{ A\ar@{^{(}->}[r]^i \ar[d] & B \ar[d] \\ C\ar[r] & D}\]

in $\spsms$ such that $A,B,C\in\Spc^{ft}$ and $i$ is a monomorphism,
we have $D\in\Spc^{ft}$.
\end{itemize}
\item[ii)] We denote by $\Spc^{ft}_\bullet\subseteq\spsms_\bullet$
the strictly full subcategory of objects $(X,x)$ such that $X\in\Spc^{ft}$.
\end{itemize}
\end{defn}

We denote by $\map_\bullet$ (resp. $\Hom_\bullet$) the mapping spaces 
(respectively internal homs) of the simplicial monoidal model
category $L_\Sh\Mh$ and their derived analogues by $\Rmap$
(resp. $\RHom_\bullet$).

\begin{theorem}\label{finsect} Let $S$ be a noetherian scheme of finite Krull
 dimension such that $Sm/S$ is countable, $F\in\spsms_\bullet$ sectionwise
 countable and
$(X,x)\in\Spc^{ft}_\bullet$. Then, 
for all $n\ge 0$, \[ \pi_n\, \Rmap_\bullet((X,x),F)\] is countable.
\end{theorem}

\begin{proof} We first show that the class

\[ \{ (X,x)\in\H_\bullet(S) \, |\, \forall n\ge 0: \pi_n\, \Rmap((X,x),F)
\mbox{ is countable }\}\subseteq\H_\bullet(S)\]

is stable under homotopy push-outs.
For this, it suffices to see that if

\[ \xymatrix{ X\ar[r]\ar[d] & X_1\ar[d]^{f_1 }\\ X_3\ar[r]^{f_3} &X_2} \]

is a homotopy pull-back of simplicial sets such that all homotopy sets of 
$X_1,X_2$ and $X_3$ are countable, then so are those of $X$.
We can assume that all $X_i$ are Kan complexes and $f_1,f_3$ are Kan
fibrations and now use some basic results about minimal Kan complexes/fibrations
\cite[Sections 9 and 10]{may}: There is a minimal Kan complex $Y_2\subseteq
X_2$ which is a deformation retract of $X_2$
and we can replace the pull-backs
of $f_i$ to $Y_2$ with minimal Kan fibrations $g_i:Y_i\to Y_2$ ($i=1,3$).
Then all $Y_i$ are minimal Kan complexes and by minimality and
our assumption about their homotopy, they are countable. Then $X$ is weakly
equivalent to the countable Kan complex $Y_1\times_{Y_2} Y_3$.\\ 

Since the identity is a simplicial (left) Quillen equivalence from $L_\Sh\Mh$
to $\spsms_\bullet$ with the $\A^1$-local injective structure, every 
$(X,x)\in\Spc^{ft}_\bullet$ can be obtained from representables by finitely many
homotopy push-outs. We may thus assume that $(X,x)=(U,u)$
for some $U\in Sm/S$. For a sectionwise countable fibrant 
replacement $F'$ of $F$ in $L_\Sh\Mh$ as in Proposition \ref{fibrep}, 
we then see that 

\[ \Rmap_\bullet((U,u),F)=\map_\bullet ((U,u),F')\subseteq\map(U,F')\simeq F'(U)\]

is a countable Kan complex.
\end{proof}

\section{Stable results}\label{stable}

Let $S$ be a noetherian scheme of finite Krull dimension and $\SH(S)$
the homotopy category of $\Pa^1$-spectra over $S$ \cite{jardine}. 
We first establish the following result \cite[Theorem 5.2]{voevodsky}.

\begin{theorem}\label{maps}
If $(X,x)\in\Spc^{ft}_\bullet$ and $E=(E_n)_{n\ge 0}$
is a $\Pa^1$-spectrum, then 

\[ \Rmap_{\SH(S)}(\Sigma^\infty_{\Pa^1}(X,x),E)\simeq\hocolim_n\Rmap_{\H_\bullet(S)}(\Sigma^n_{\Pa^1}(X,x),E_n).\]

In particular,

\[ \SH(S)(\Sigma^\infty_{\Pa^1}(X,x),E)=\colim_n\, \H_\bullet(S)(\Sigma^n_{\Pa^1}
(X,x), E_n).\]

\end{theorem}

\begin{proof} 
Using an argument very similar to the first part of the proof of 
Theorem \ref{finsect} powered by the facts that $\Sigma^\infty_{\Pa^1}$
and $\Sigma^n_{\Pa^1}$ preserve homotopy push-outs and that filtered
(homotopy) colimits commute with finite (homotopy) limits in simplicial sets,
we can assume that $(X,x)=(U,u)$ for some $U\in Sm/S$.\\
We now check the hypothesis of \cite[Corollary 4.13]{hovey}:
${\mathcal D}:=L_\Sh\Mh$ is (left) proper, cellular and almost finitely generated
, $T:=(\Pa^1,\infty)\wedge -:{\mathcal D}\to{\mathcal D}$ is left Quillen
since ${\mathcal D}$ is monoidal and the right-adjoint of $T$ is $U=\Omega_{\Pa^1}$
and preserves filtered colimits. 
For every $m\ge 0$, $A:=(U,u)\wedge S^m\in{\mathcal D}$ is compact and cofibrant
with compact cylinder object $A\wedge \Delta^1_+$. Finally, let $Y$ denote a 
level fibrant replacement of $E$, then 

\[ \pi_m\Rmap_{\SH(S)}(\Sigma^\infty_{\Pa^1}(U,u),E)=\SH(S)(\Sigma^\infty_{\Pa^1}(A),Y)=\]

\[ \colim_n \H_\bullet(S)(\Sigma^n_{\Pa^1}(A),Y_n)=\pi_m(\hocolim_n\Rmap_{\H_\bullet(S)}(\Sigma^n_{\Pa^1}(U,u),E_n)).\]

Here, the second equality is \cite[Corollary 4.13]{hovey} and we used
\cite[Theorem A.5.6]{ppr} to know that $\Ho(Sp^\N({\mathcal D},T))=
\SH(S)$.
\end{proof}

Let $s_-$ denote the shift functor for $\Pa^1$-spectra, i.e. $(s_-(E))_n=E_{n+1}$.
We need the following observation about its homotopical 
properties.

\begin{lemma}\label{dersmash}
Let $S$ be a noetherian scheme of finite Krull dimension and $E=(E_n)$ 
a $\Pa^1$-spectrum. Then

\[ \mathbb{L}((\Pa^1,\infty)\wedge -)(E)\simeq s_-(E)\mbox{ in }\SH(S).\]

\end{lemma}

\begin{proof}
We use somewhat freely results and notations from 
\cite[Sections 3 and 4]{hovey}. We know that $\mathbb{L}((\Pa^1,\infty)\wedge -)
\simeq\mathbb{R}s_-$. Construct

\[ \xymatrix{ E\ar[r]^\iota& E'\ar[r]^{j_{E'}} & \Theta^\infty E',}\]

where $\iota$ is a level-fibrant replacement and $j_{E'}$ and $\Theta^\infty$
are as in {\em loc. cit.}\\
Then $\Theta^\infty E'$ is stably fibrant, hence

\[ \mathbb{R}s_-(E)\simeq s_-\Theta^\infty E'\simeq \Theta^\infty s_- E',\]

since $s_-\circ\Theta\simeq\Theta\circ s_-$ by direct inspection.
On the other hand, $s_-(\iota):s_-E\to s_-E'$ is a level-equivalence and
thus a stable equivalence, and $j_{s_-E'}:s_-E'\to \Theta^\infty s_-E'$ was 
shown to be a stable weak equivalence in {\em loc. cit.} Combining, we
see that $\mathbb{R}s_-(E)\simeq s_-E$, as desired.
\end{proof}

To resume our work on Brown representability, recall that

\[ \Gh:=\left\{ \Sigma^{p,q}\Sigma^\infty_{\Pa^1} U_+\, |\, U\in Sm/S, p,q\in\Z \right\}\subseteq\SH(S)\]

is a set of compact generators, where $\Sigma^{p+q,q}:=S^p\wedge\G_m^q\wedge (-)$
(for $p,q\ge 0$). 
Denoting by $SH(S)^c\subseteq SH(S)$
the full subcategory of compact objects, we first deduce the stable analogue
of Theorem \ref{finsect}.

\begin{theorem}\label{stablefinsect}
Let $S$ be a noetherian scheme of finite Krull dimension
such that $Sm/S$ is countable, $F\in\SH(S)^c$ and
$E=(E_n)$ a $\Pa^1$-spectrum such that for all $n\ge 0$, $E_n$ is sectionwise
countable. Then, for all $p,q\in\Z$, 

\[ E^{p,q}(F):=\SH(S)(F,\Sigma^{p,q}E)\]

is countable.
\end{theorem}

\begin{proof} It is clear that 
the class \[ \{ F'\in\SH(S)\, |\, \forall p,q\in\Z:E^{p,q}(F) \mbox{ is 
countable }\}\subseteq\SH(S)\] 
is a thick subcategory stable under $S^1\wedge -$ and $(\G_m,1)\wedge -$.\\
$\SH(S)^c\subseteq\SH(S)$ is the thick subcategory generated by $\Gh$
\cite[Theorem 2.1.3]{neemangrothendieck}, so we can assume that $F=
\Sigma^\infty_{\Pa^1}U_+$ for some $U\in Sm/S$.\\
Given $p,q\in\Z$, choose integers $p',q',k\ge 0$ such that 
$(p,q)=(p'+q',q')-(2k,k)$. Then 

\[ E^{-p,-q}(F)=\SH(S)(\Sigma^\infty_{\Pa^1}(
U_+\wedge S^{p'}\wedge\G_m^{q'}),\mathbb{L}((\Pa^1,\infty)\wedge -)^k(E)).\]

By Lemma \ref{dersmash}, there is a stable equivalence
$\mathbb{L}((\Pa^1,\infty)\wedge -)^k(E)\simeq (s_-)^k(E)=:E'$. Since 
$(X,x):=U_+\wedge S^{p'}\wedge\G_m^{q'}$ is isomorphic
in $\H_\bullet(S)$ to a space of finite type 
(because $S^1\simeq \A^1/(0\sim 1)$), Theorem \ref{maps}
implies that

\[ E^{-p,-q}(F)=\colim_n \HH_\bullet(S)(\Sigma^n_{\Pa^1}(X,x),E'_n).\]

Since $(X,x)\in L_\Sh\Mh$ is cofibrant, we have 

\[ X_n:=\Sigma^n_{\Pa^1}(X,x)=(\Pa^1,\infty)^{\wedge n}\wedge (X,x),\]

which again is isomorphic in $\H_\bullet(S)$ to a space of finite type.
Now, for every $n\ge 0$, $\HH_\bullet(S)(X_n,E'_n)$ is countable by Theorem
\ref{finsect}, hence so is $E^{-p,-q}(F)$.
\end{proof}

Finally, we can establish \cite[Proposition 5.5]{voevodsky} which makes
Brown representability available in $\A^1$-homotopy theory.

\begin{theorem}\label{brown} 
Let $S$ be a noetherian scheme of finite Krull dimension such 
that $Sm/S$ is countable.\\
Then the category $\SH(S)^c$ is countable.
\end{theorem}

\begin{proof} According to the proof of
\cite[Proposition 2.3.5]{hps}, $SH(S)^c$ is equivalent to an increasing union of subcategories

\[ \Sh_0\subseteq\Sh_1\subseteq\ldots\subseteq\SH(S)^c\subseteq\SH(S)\]

such that $\Sh_0$ is the full subcategory spanned by the set of objects $\Gh$ and
for every $n\ge 0$, if $\Sh_n$ is countable, so is $\Sh_{n+1}$. It thus suffices to see that $\Sh_0$ is countable. Since $\Gh$ clearly is, this means we need
to show that for all $p,q\in\Z, U,V\in Sm/S$ the set 

\[ \SH(S)(\Sigma^{\infty}_{\Pa^1} U_+,\Sigma^{p,q}\Sigma^\infty_{\Pa^1} V_+)\] 

is countable, which is true by Theorem \ref{stablefinsect}
because for all $n\ge 0$, $E_n:=(\Sigma^\infty_{\Pa^1}V_+)_n=
(\Pa^1,\infty)^{\wedge n}\wedge V_+$ is sectionwise countable.
Indeed, for all $U\in Sm/S$, the (simplicial)
set $E_n(U)$ is a quotient of 

\[ [(Sm/S)(U,\Pa^1)]^n\times (Sm/S)(U,V)=\Pa^1(U)^n\times V(U)\]
which is countable since the category $Sm/S$ is.
\end{proof}

\end{document}